\documentclass[11pt]{article}

\usepackage[letterpaper,twoside,outer=1.3in,vmargin=1.3in,]{geometry}

\usepackage{amsmath}
\usepackage{amsthm}
\usepackage{amsfonts}
\usepackage{amssymb}
\usepackage{times}
\usepackage{epsfig}
\usepackage{xcolor}
\usepackage{hyperref}
\usepackage{cleveref}
\usepackage{enumerate}

\renewcommand{\baselinestretch}{1.3}

\DeclareMathOperator{\aff}{aff}
\DeclareMathOperator{\spn}{span}
\DeclareMathOperator{\row}{row}
\DeclareMathOperator{\rec}{rec}
\DeclareMathOperator{\lin}{lin}
\DeclareMathOperator{\cone}{cone}

\newcommand{\1}{\mathbf{1}}
\newcommand{\0}{\mathbf{0}}
\newcommand{\zP}{\mathcal P}
\newcommand{\zC}{\mathcal C}
\newcommand{\cL}{\mathbb L}
\newcommand{\cS}{\mathcal S}
\newcommand{\R}{\mathbb R}
\newcommand{\Z}{\mathbb Z}
\newcommand{\ii}{\hat{\i}}

\newtheorem{theorem}{Theorem}[section]
\newtheorem{CO}[theorem]{Corollary}

\newtheorem{PR}[theorem]{Proposition}

\newtheorem{RE}[theorem]{Remark}
\newtheorem{example}[theorem]{Example}

\newtheorem*{QU*}{Questions}

\newcounter{claim_nb}[theorem]
\newtheorem{claim}[claim_nb]{Claim}
\newtheorem*{claim*}{Claim}

\newcounter{claim_nbs}[section]
\newcounter{subclaim_nb}[claim_nbs]
\newenvironment{cproof}
{\begin{proof}
 [Proof of Claim.]
 \vspace{-1.2\parsep}}
{ \end{proof}}

\newtheoremstyle{named}{}{}{\itshape}{}{\bfseries}{.}{.5em}{\thmnote{#3 }#1}
\theoremstyle{named}

\title{Generalizations of Total Dual Integrality\thanks{This research was supported in part by Discovery
Grants from NSERC.}}

\author{
Bertrand Guenin
\thanks{Department of Combinatorics and Optimization, University of Waterloo, email: bguenin@uwaterloo.ca}
\and
Levent Tun\c{c}el
\thanks{Department of Combinatorics and Optimization, University of Waterloo, email: levent.tuncel@uwaterloo.ca}}

\begin{document}

\maketitle

\begin{abstract}
We design new tools to study variants of Total Dual Integrality.
As an application, we obtain a geometric characterization of Total Dual Integrality for the case where the associated polyhedron is non-degenerate.
We also give sufficient conditions for a system to be Totally Dual Dyadic, and prove new special cases of Seymour's Dyadic conjecture on ideal clutters. 
\end{abstract}


\section{Introduction}\label{sec-intro}
Consider a system of linear inequalities $Mx\leq b$ where $M,b$ are integral and $M$ has $m$ rows and $n$ columns.
For a vector $w\in\R^n$ define the following primal-dual pair of linear programs,
\begin{align*}
& \max\{w^\top x: Mx\leq b\} \tag{$P:M,b,w$}, \\
& \min\{b^\top y: M^\top y=w, y\geq\0\} \tag{$D:M,b,w$}.
\end{align*}
Recall that from Strong Duality, both or none of these linear programs have an optimal solution.
We say that $w$ is {\em admissible} in the former case.
Let $\cL$ denote a subset of the reals $\R$.
We say that $Mx\leq b$ is {\em totally dual in $\cL$} (abbreviated as TD in $\cL$) if for every $w\in\cL^n$ for which $w$ is admissible, $(D\!:\!M,b,w)$ has an optimal solution in $\cL^m$.
When $\cL=\Z$ this corresponds to totally dual integral (TDI) systems.
We say that $\cL\subseteq\R$ is {\em heavy} if it is a dense subset of $\R$ that contains all integers and $(\cL,+)$ forms a subgroup of the additive group of real numbers.
Our key result, \Cref{main-char}, yields a characterization of systems that are totally dual in $\cL$ when $\cL$ is heavy.
To motivate this result, we present in the introduction applications, and highlights their relevance to the study of TDI systems and to a long standing conjecture of Paul Seymour on ideal clutters.
\subsection{Totally dual in $\cL$ systems}\label{sec-tdl}
Consider a prime $p$.
A rational number is {\em $p$-adic} if it is of the form $\frac{r}{s}$ where $r,s\in\Z$ and $s$ is an integer power of $p$  \cite{abdi23b}.
A number is {\em dyadic} if it is $2$-adic.
Let $\cS$ be a (possibly infinite) set of primes.
We denote by $\cL(\cS)$ the set of all rationals of the form $\frac{r}{s}$ where $r,s\in\Z$ and $s$ is a product of integer powers of primes in $\cS$.
Observe that $\cL({\{p\}})$ denotes the $p$-adic rationals. 
To keep the notation light, we write $\cL_p$ for $\cL({\{p\}})$.
For any set of primes $\cS$, $\cL(\cS)$ is a heavy set (this follows from \cite[Lemma 2.1]{abdi23}).

For a system $Mx\leq b$ to be TD in $\cL$ we consider the dual $(D\!:\!M,b,w)$ for all choices of $w\in\cL^n$.
However, for the aforementioned heavy sets, it suffices to consider the choices $w\in\Z^n$. Namely,
\begin{RE}\label{restricted-w}
Let $\cL:=\cL(\cS)$ where $\cS$ is a set of primes.
Then $Mx\leq b$ is TD in $\cL$ if and only if for every admissible $w\in\Z^n$, $(D\!:\!M,b,w)$ has an optimal solution in $\cL^m$.
\end{RE}
\begin{proof}
Necessity is clear, let us prove sufficiency.
Consider an admissible $w\in\cL^n$.
For some $\mu$ that is a product of integer powers of primes in $\cS$ we have $\mu w\in\Z^n$.
There exists an optimal solution $\bar{y}\in\cL^n$ for $(D\!:\!M,b,\mu w)$.
However, then $\frac1\mu\bar{y}$ is a solution for $(D\!:\!M,b,w)$ that is in $\cL^n$.
\end{proof}
For any prime $p$, we can find a system $Mx\leq b$ that is TD in $\cL_p$ but not TD in $\cL_q$ for any prime $q\neq p$.
Namely, one can pick $Mx\leq b$ to consist of a unique constraint, $px\leq 1$. 
However, if we require a system to be TD in $\cL_p$ and TD in $\cL_{p'}$ for distinct primes $p$ and $p'$, 
then it is totally dual in $\cL_q$ for every prime $q$ \cite[Theorem 1.4]{abdi23b}.
The following stronger statement (the aforementioned case corresponds to $\cS_1=\{p\}, \cS_2=\{p'\}$ and $k=2$) holds in the full dimensional case,
\begin{theorem}\label{pq-to-L}
Let $Mx\leq b$ be a system where $\{x:Mx\leq b\}$ is a full-dimensional polyhedron.
For $i=1,\ldots,k$, let $\cS_i$ be a set of primes and suppose that $Mx\leq b$ is TD in $\cL(\cS_i)$.
If $\cap_{i\in[k]}\cS_i=\emptyset$ then $Mx\leq b$ is TD in $\cL_q$ for every prime $q$.
\end{theorem}
\noindent
This will be an immediate consequence of \Cref{main-char}.

Consider an integral matrix $M$ and an integral vector $b$.
A necessary condition for $Mx\leq b$ to be TDI is that the polyhedron $Q=\{x\in\R^n:Mx\leq b\}$ be integral, 
i.e., that every minimal proper face of $Q$ contains an integral vector \cite{Edmonds77}, see also \cite[Corollary 22.1a]{schrijver86}.
However, this is not a sufficient condition \cite[Equation (3) in Chapter 22]{schrijver86}.
Let us define a stronger necessary condition for a system to be TDI.
We say that $Mx\leq b$ is {\em near-TDI} if for every prime $p$, $Mx\leq b$ is TD in $\cL_p$.
Since $\Z\subset\cL_p$ for every prime $p$, it then follows from \Cref{restricted-w} that if a system is TDI, it is near-TDI.

Furthermore, we have the following result \cite[Theorem 1.5]{abdi23},
\begin{PR}\label{pq-to-integral}
Consider $M\in\Z^{m\times n}$ and $b\in\Z^m$.
If $Mx\leq b$ is TD in $\cL_p$ and TD in $\cL_q$ for distinct primes $p,q$ then the polyhedron $\{x:Mx\leq b\}$ is integral.
\end{PR}
\noindent
\Cref{main-char} characterizes when a system $Mx\leq b$ is TD in $\cL_p$ for some prime $p$. The proof leverages the density of $\cL_p$.
This suggests, unexpectedly, that density arguments may sometimes be very useful in certifying integrality of polyhedra via \Cref{pq-to-integral}.

It follows from \Cref{pq-to-integral}  that if $Mx\leq b$ is near-TDI then $\{x:Mx\leq b\}$ is integral.
Therefore, for $M,b$ integer, if $Mx\leq b$ is TDI then it is near-TDI, and if $Mx\leq b$ is near-TDI, then $\{x:Mx\leq b\}$ is integral.
\begin{figure}[htb!]
\centering\includegraphics[scale=0.65]{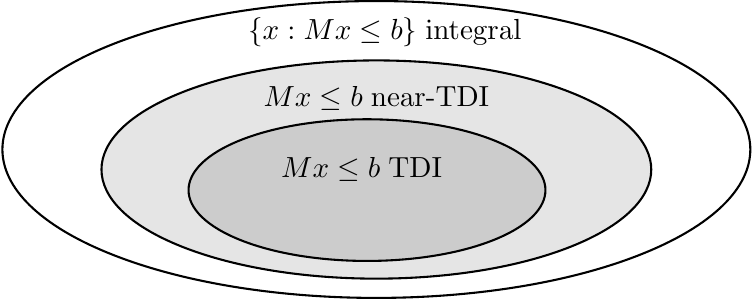}
\caption{Hierarchy of linear systems of inequalites.}\label{fig-wenn}
\end{figure}

A system $Mx\leq b$ is {\em non-degenerate} if for every minimal proper face $F$ of $\{x:Mx\leq b\}$
the left-hand sides of the tight constraints defining $F$, are linearly independent with the caveat 
that if we have two tight constraints of the form $\alpha^\top x\leq\beta$ and $-\alpha^\top x\leq-\beta$ we only consider one of these as defining $F$.
For non-degenerate systems, the TDI and near-TDI properties coincide. 
Namely,
\begin{RE}\label{non-deg-near}
Every non-degenerate near-TDI system is TDI.
\end{RE}
\begin{proof}
Let $Mx\leq b$ be non-degenerate and let $w\in\Z^n$ be admissible weights.
Let $F$ denote the set of optimal solutions to $(P\!:\!M,b,w)$.
Let $I$ be the set of constraints that defines the face $F$.
Let $\bar{y}$ be an optimal solution to $(D\!:\!M,b,w)$.
By complementary slackness, $\bar{y}_i=0$ for all $i\notin I$.
Since $Mx\leq b$ is non-degenerate, $\bar{y}$ is the unique solution to $M^\top y=w$ and $y_i=0$ for all $i\notin I$.
As $Mx\leq b$ is near-TDI, it is TD in $\cL_p$ and TD in $\cL_q$ for distinct primes $p$ and $q$.
It follows that $\bar{y}\in\cL_p^m\cap \cL_q^m=\Z^m$.
Hence, $(D\!:\!M,b,w)$ has an integer optimal solution.
\end{proof}
\noindent
In the next section, we give a geometric characterization of non-degenerate TDI systems.
\subsection{A geometric characterization of non-degenerate TDI systems}
Consider a matrix $M\in\Z^{m\times n}$ and a vector $b\in\Z^m$.
The system $Mx\leq b$ is {\em resilient} if $Q:=\{x \in \R^n :Mx\leq b\}$ is integral and for every $i\in[m] :=\{1,2, \ldots,m\}$, 
\[
Q\cap\{x \in \R^n : \row_i(M)x\leq b_i-1\},
\]
is also an integral polyhedron. 
To illustrate this property consider, 
\begin{align}
\left(
\begin{array}{rr}
1 & 1 \\
-1 & 0\\
0 & -1
\end{array}
\right)
x
&
\leq
\begin{pmatrix}
3 \\ 0 \\ 0
\end{pmatrix},\quad\mbox{and}\quad
\label{resilient-ex1} \\
\left(
\begin{array}{rr}
3 & 1 \\
0 & 1\\
-1 & 0\\
0 & -1
\end{array}
\right)
x
&
\leq
\begin{pmatrix}
6 \\ 3 \\ 0 \\ 0
\end{pmatrix}.
\label{resilient-ex2}
\end{align}
Denote by $Q$ the polyhedron defined by \eqref{resilient-ex1} and by $Q'$ the polyhedron defined by \eqref{resilient-ex2}, see \Cref{fig-resilient}.
Both $Q$ and $Q'$ are integral; however, $Q$ is resilient, but $Q'$ is not.
\begin{figure}[htb!]
\vspace{-0.1in}
\centering\includegraphics[scale=0.35]{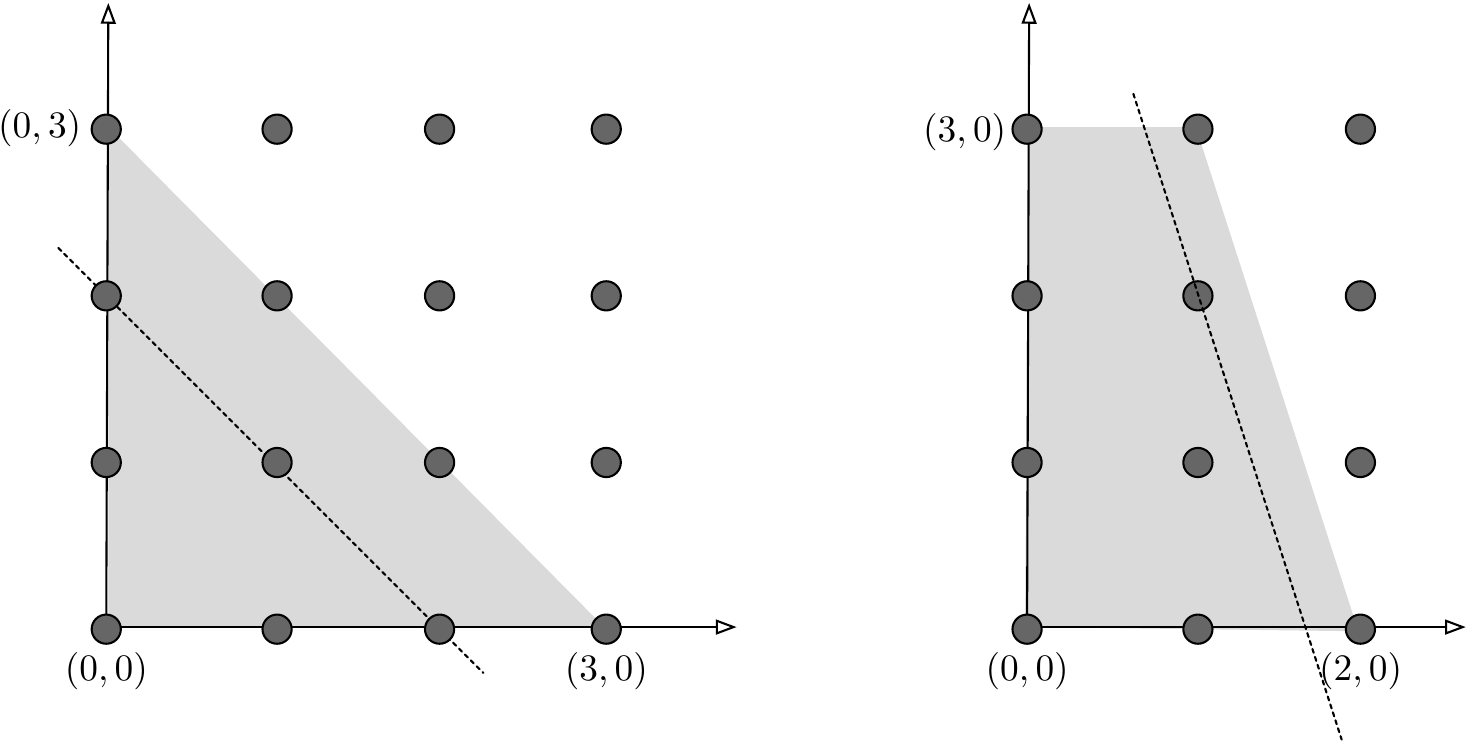}
\caption{Left $Q$ and a shifted hyperplane, right $Q'$ and a shifted hyperplane.}\label{fig-resilient}
\end{figure}

We prove the following result (see \Cref{gen-tdi} in \S\ref{sec-tdi}),
\begin{theorem}\label{intro-tdi}
Consider $M\in\Z^{m\times n}$ and $b\in\Z^m$.
Suppose that $\{x \in \R^n : Mx\leq b\}$ is full-dimensional and $Mx\leq b$ is non-degenerate.
Then $Mx\leq b$ is TDI if and only if $Mx\leq b$ is resilient.
\end{theorem}
For $M\in\Z^{m\times n}$ and $b\in\Z^m$ we say that $Mx\leq b$ is a {\em faceted system} if 
(i) constraints in $Mx\leq b$ are irredundant (i.e. removing any constraint creates new solutions),
(ii) for any constraint, the greatest common divisor (gcd) of the entries on the LHS and RHS is one, and
(iii) $\{x\in\R^n:Mx\leq b\}$ is full-dimensional.
Recall that for a full-dimensional rational polyhedron $P$, there is, up to scaling, 
a unique integral system with irredundant constraints that describes $P$ \cite[Corollary 3.31]{CCZ2014}.
It follows that there is a bijection between faceted systems and full-dimensional rational polyhedra.
In particular, faceted systems can be viewed as geometric objects.
Consider an integral matrix $M$ and an integral vector $b$ where $Mx\leq b$ is TDI.
Suppose that for some constraint $i$, the gcd of the LHS and RHS is $\alpha\geq 2$.
Then, the system obtained from $Mx\leq b$ by dividing constraint $i$ by $\alpha$, is also TDI.
Hence, when studying TDI systems there is no loss of generality in assuming condition (ii) holds.
However, conditions (i) and (iii) are actual restrictions.

Next we give a geometric interpretation of resiliency for faceted systems.
\begin{PR}\label{geometric}
Let $Mx\leq b$ be a faceted system which defines an integral polyhedron $Q$.
Then $Mx\leq b$ is resilient if and only if for every facet $F$ of $Q$
the polyhedron $Q_F$ is integral where $Q_F$ is obtained from $Q$ by shifting-in the supporting hyperplane for $F$ to the next integer lattice point.
\end{PR}
\noindent
Consider for instance the polyhedron $Q$ described by \eqref{resilient-ex1} and illustrated in \Cref{fig-resilient}.
Let $F$ be the face of $Q$ determined by the supporting hyperplane $H=\{x\in\R^2:x_1+x_2=3\}$.
Shifting-in $H$ (i.e. decreasing the RHS) until the hyperplane contains a new lattice point, yields $H'=\{x\in\R^2:x_1+x_2=2\}$.
Then $Q_F$ is the polyhedron described by \eqref{resilient-ex1} where we change $b_1=3$ to $b_1=2$.
Before we proceed with the proof, recall the following immediate consequence of Bezout's lemma,
\begin{RE}\label{easy} 
$\alpha^\top x=\beta$ where $\alpha\in\Z^n, \beta\in\Z$ has solution in $\Z^n$
if and only if $\gcd(\alpha_1,\ldots,\alpha_n) | \beta$.
\end{RE}
\begin{proof}[Proof of \Cref{geometric}]
Let $m,n$ denote the number of rows, respectively, columns of $M$.
There exists a bijection between constraints $i\in[m]$ of $Mx\leq b$ and facets $F_i$ of $Q$ 
where $F_i=Q\cap H_i$ and $H_i=\{x:\row_i(M)x=b_i\}$ is a supporting hyperplane of $Q$ \cite[Theorem 3.27]{CCZ2014}.

Pick $i\in[m]$ and let $H'_i=\{x:\row_i(M)x=b_i-1\}$.
\begin{claim*}
$H'_i\cap\Z^n\neq\emptyset$.
\end{claim*}
\begin{cproof}
Since $Q$ is integral, the face $F_i\subseteq H_i$ contains an integral point.
Hence, $\row_i(M)x=b_i$ has an solution in $\Z^n$.
\Cref{easy} implies that $\mu:=\gcd(M_{i1},\ldots,M_{in})|b_i$.
Since $Mx\leq b$ is faceted, we have $\mu=1$.
Therefore, by \Cref{easy} $\row_i(M)x=b_i-1$ has an integer solution.
\end{cproof}
\noindent
Since there is no integer between $b_i-1$ and $b_i$ we have $\{x:b_i-1<\row_i(M)x<b_i\}\cap\Z^n=\emptyset$.
This implies with the claim that $H'_i$ is obtained by shifting-in $H_i$ to the next integer lattice point.
\end{proof}
\Cref{non-deg-near} says that non-degenerate near-TDI systems are TDI.
Moreover, if in addition the system in faceted, then we have a geometric description of TDI via \Cref{intro-tdi} and \Cref{geometric}.
This motivates the following problems,
\begin{enumerate}[\;\;-]
\item Can we characterize when a near-TDI faceted system is TDI?
\item Can we find a geometric characterization of when a faceted system is TDI?
\end{enumerate}
\noindent
In closing, we will observe that each of the inclusions in \Cref{fig-wenn} are strict.
Consider the system $Mx\leq b$ described in \eqref{resilient-ex2}.
The associated polyhedron $\{x:Mx\leq b\}$ is integral.
However, $Mx\leq b$ is not resilient.
Therefore, by \Cref{intro-tdi}, $Mx\leq b$ is not TDI.
Since $Mx\leq b$ is non-degenerate, it then follows from \Cref{non-deg-near} that $Mx\leq b$ is not near-TDI either.
Consider now the system,
\begin{equation}\label{eq-w1}
\begin{pmatrix} 2\\3 \end{pmatrix}x\leq\begin{pmatrix}0 \\ 0\end{pmatrix},
\end{equation}
with the associated dual problem $(D\!:\!M,b,w)$,
\begin{equation}\label{eq-w2}
\min\{0:2y_1+3y_2=w, y_1,y_2\geq 0\}.
\end{equation}
$w$ admissible iff $w\geq 0$. If $w$ is admissible and not equal to zero, then $x=0$ is the unique optimal solution for $(P\!:\!M,b,w)$.
Observe that $\{2y_1+3y_2: y_1,y_2\in\Z_+\}=\Z\setminus\{1\}$.
This implies that \eqref{eq-w2} has an $p$-adic solution for every prime $p$ and every choice $w\in\Z_+$ but that \eqref{eq-w2} does not have an integral solution when $w=1$.
In particular, \eqref{eq-w1} is near-TDI but not TDI.
Note that \eqref{eq-w1} is not a faceted system, however.
\subsection{Totally dyadic systems}
A system is totally dual dyadic (abbreviated as TDD) if it is TD in $\cL(\{2\})$, i.e. for every integer admissible $w$, $(D\!:\!M,b,w)$ has an optimal solution that is dyadic.
We find a sufficient condition for a system of linear inequalities to be TDD by relaxing the resilient condition.
Namely, let us say that a system $Mx\leq b$ (where $M\in\Z^{m\times n},b\in\Z^m$) is {\em half-resilient} if $Q:=\{x \in \R^n : Mx\leq b\}$ is integral and for every $i\in[m]$, 
\[
Q\cap\{x \in \R^n : \row_i(M)x\leq b_i-1\} \quad\mbox{or}\quad Q\cap\{x \in \R^n : \row_i(M)x\leq b_i-2\}
\]
is also an integral polyhedron. Clearly, every resilient system is half-resilient. 

We prove the following result (see \Cref{sufficient-res-co} part (b)),
\begin{theorem}\label{intro-TDD}
Consider matrix $M\in\Z^{m\times n}$ and vector $b\in\Z^m$.
Suppose that $Q=\{x \in \R^n : Mx\leq b\}$ is full-dimensional. 
If $Mx\leq b$ is half-resilient then $Mx\leq b$ is TDD.
\end{theorem}

The study of TDD systems was initially motivated by a conjecture on ideal clutters.
A {\em clutter} $\zC$ is a family of sets over ground set $E(\zC)$ with the property that no distinct pair $S,S'\in\zC$ satisfies $S\subseteq S'$.
Given a clutter $\zC$, we define the $0,1$ matrix $T(\zC)$ where the rows of $T(\zC)$ correspond to the characteristic vectors of the sets in $\zC$.
A clutter $\zC$ is {\em ideal} if the polyhedron $\{x\geq\0: T(\zC)x\geq\1\}$ is integral.
Seymour  \cite[\S 79.3e]{schrijver03c}, proposed the following conjecture,

\vspace{0.1in}\noindent{\bf The Dyadic Conjecture.} \\
{\em Let $\zC$ be an ideal clutter, then $T(\zC)x\geq\1,x\geq\0$ is TDD.}

\vspace{0.1in}\noindent
This conjecture is known to hold when $\zC$ is the clutter of any one of the following families:
(a) odd cycles of a graph \cite{guenin2001,geelen02}, 
(b) $T$-cuts of a graft \cite{lovasz75},
(c) $T$-joins of a graft \cite{abdi22d},
(d) dicuts of a directed graph \cite{lucchesi78},
(e) dijoins of a directed graph \cite{gueninH2024}.
Moreover, when $\zC$ is ideal, $w\in\Z^{E(\zC)}_+$ and $\min\{w^\top x: T(\zC)x\geq\1, x\geq\0\}=2$ then the dual has an optimal solution that is dyadic \cite{abdi-dyadic}.
For both (a) and (b) the system $T(\zC)x\geq\1,x\geq\0$ is in fact $\frac12$-TDI.
In this paper, we present a result based on the intersection of members of the clutter and its blocker. 
A set $B\subseteq E(\zC)$ is a {\em cover} of $\zC$ if $B\cap S\neq\emptyset$ for every $S\in\zC$.
The set of all inclusion-wise minimal covers of a clutter $\zC$ forms another clutter called the {\em blocker} of $\zC$.
We denote by $b(\zC)$ the blocker of $\zC$.

We prove the following result (see \Cref{dyadic-co}(a)),
\begin{theorem}\label{intro-inter3}
Let $\zC$ be an ideal clutter. 
If for every $S\in\zC$ and every $B\in b(\zC)$ we have $|S\cap B|\leq 3$ then $T(\zC) x \geq\1,x\geq\0$ is TDD.
\end{theorem}
\noindent
Note, the following related result \cite[Theorem 2]{cornuejols00}.
\begin{theorem}
Let $\zC$ be an ideal clutter. 
If for every $S\in\zC$ and every $B\in b(\zC)$ we have $|S\cap B|\leq 2$ then $T(\zC) x \geq\1,x\geq\0$ is TDI.
\end{theorem}
\subsection{Organization of the paper}
In \S\ref{sec-gen-char} we give a characterization of when a system is TD in $\cL$ for a heavy set $\cL$.
Since the statement is a bit technical, we postpone the proof of that result until \S\ref{sec-gen-proof} and first focus on applications.
In \S\ref{sec-applications} we derive sufficient conditions for a system to be TD in $\cL$ for various choices of $\cL$ and explain the role of resiliency and half-resiliency.
\S\ref{sec-tdi} leverages these results to characterize non-degenerate TDI systems.
A special case of Seymour's Dyadic Conjecture is proved in \S\ref{sec-seymour}.
\section{The key result}\label{sec-gen-char}
In this section we present a characterization of when a system $Mx\leq b$ is Totally Dual in $\cL$ for the case where $\cL$ is heavy.
First, we need to present the central notion of the \emph{tilt constraints}.
\subsection{The tilt constraint}
For a subset $S$ of $\R^n$, $\aff(S)$ denotes the {\em affine hull} of $S$, i.e., the smallest affine space that contains~$S$.
Consider a system $Mx\leq b$ where $M$ is an $m\times n$ matrix.
Constraint $i\in[m]$ of $Mx\leq b$ is {\em tight} for some $x\in\R^n$ if $\row_i(M)x=b_i$.
Constraint $i\in[m]$ of $Mx\leq b$ is {\em tight} for some $S\subseteq\R^n$ if it is tight for every $x\in S$.
The {\em implicit equalities} of $Mx\leq b$ are the constraints of $Mx\leq b$ that are tight for $Q:=\{x \in \R^n : Mx\leq b\}$.
Given a set $S\subseteq\R^n$ we denote by $I_{M,b}(S)\subseteq[m]$ the index set of tight constraints of $Mx\leq b$ for $S$.
We abuse the terminology and say that $F$ is a face of the system $Mx\leq b$ to mean that $F$ is a face of $Q$.
Given faces $F,F^+$ of $Q$ we say that $F^+$ is a {\em down-face} of $F$, if $F\subset F^+$ and $\dim(F^+) = \dim(F)+1$.

Consider now nonempty faces $F,F^+$ of $Mx\leq b$ where $F^+$ is a down-face of $F$ and assume that $F$ is defined by a supporting hyperplane 
$\{x \in \R^n : w^\top x=\tau\}$ for some (non-zero) vector $w\in\R^n$ and some $\tau\in\R$, i.e.,
$F=Q\cap\{x \in \R^n : w^\top x=\tau\}$ and $Q \subseteq \{x \in \R^n\, : \, w^{\top}x \leq \tau\}$.
Pick, $\rho\in \aff(F^+)\setminus\aff(F)$ and define,
\begin{equation}\label{tilt-eq}
\frac1{\tau-w^\top\rho}\sum\Bigl([b_i-\row_i(M)\rho]u_i : i\in I_{M,b}(F)\setminus I_{M,b}(F^+)\Bigr)=1.
\end{equation}
This is a constraint of the form, $\alpha^\top u=1$, where $\alpha$ is determined by $w,F,F^+,\tau,\rho$.
It defines a hyperplane in the space of the variables $u_i$.
In our applications, we will be given $\alpha$ and will be interested in knowing if there exists $u_i\in\cL$ that satisfy $\alpha^\top u=1$ for some suitable choice of $\cL$.
Recall, that $F$ and $F^+$ are determined by their tight constraints, i.e. $F=Q\cap\{x \in \R^n : \row_i(M)x=b_i, i\in I_{M,b}(F)\}$ and $F^+=Q\cap\{x \in \R^n : \row_i(M)x=b_i, i\in I_{M,b}(F^+)\}$.
Therefore, since $F\subset F^+$ we have $I_{M,b}(F)\setminus I_{M,b}(F^+)\neq\emptyset$.
It follows that \eqref{tilt-eq} has a least one variable $u_i$ (with a nonzero coefficient).
In particular, \eqref{tilt-eq} always has a solution with all variables $u_i$ in $\R$.

We say that \eqref{tilt-eq} is the {\em $(w,F,F^+)$-tilt constraint for the system $Mx\leq b$} (we omit specifying $Mx\leq b$ when clear from the context).
We do not refer to $\tau$ as $w^\top x=\tau$ will be a supporting hyperplane of $Q$ exactly 
when $w$ is admissible and $\tau$ is the optimal value of $(P\!:\!M,b,w)$, thus $\tau$ is determined by $w$ and $Mx\leq b$.
We do not refer to $\rho$ as \eqref{tilt-eq} is independent of the choice of $\rho \in \aff(F^+)\setminus \aff(F)$ as we will see in \Cref{ind-rep}.
\begin{example}\label{ex-part1}
Consider,
\begin{equation}\label{ex-ridge}
M=
\left(
\begin{array}{rr}
 1 & 1\\
-1 & 0\\
 0 & -1
\end{array}
\right)
\quad
b=
\begin{pmatrix}
3 \\ 0 \\ 0
\end{pmatrix},
\quad
w=
\begin{pmatrix}
0 \\ 1
\end{pmatrix}
\quad\mbox{and}\quad
\tau=3.
\end{equation}
Note that $Q:=\{x\in\R^n:Mx\leq b\}$ is given in \Cref{fig-resilient}.
Observe that $Q\cap\{x  : w^\top x=\tau\}$ contains a unique point $(0,3)$ that forms a face $F$ of $Mx\leq b$.
Let $F^+_1$ be the line segment with ends $(0,3)$ and $(3,0)$.
Let $F^+_2$ be the line segment with ends $(0,3)$ and $(0,0)$.
Then for $i=1,2$, $F^+_i$ is a down-face of $F$.
Consider $\rho_1=(3,0)$ and $\rho_2=(0,0)$.
We have for $i=1,2$, $\rho_i\in F^+_i\setminus F$.
Note, $I_{M,b}(F)=\{1,2\}$, $I_{M,b}(F^+_1)=\{1\}$, and $I_{M,b}(F^+_2)=\{2\}$.
Then the $(w,F,F^+_1)$-tilt and $(w,F,F^+_2)$-tilt constraints are respectively,
\begin{align*}
\frac{1}{\tau-w^\top\rho_1}[b_2-row_2(M)\rho_1]u_2 = 1 & \quad\iff\quad \frac33u_2=1,\\
\frac{1}{\tau-w^\top\rho_2}[b_1-row_1(M)\rho_2]u_1 = 1 & \quad\iff\quad \frac33u_1=1.
\end{align*}
\end{example}
\subsection{The characterization}\label{sec-gsc-def}
Consider a set $S=\{a^1,\ldots,a^m\}$ of integer vectors in $\Z^n$.
Then $S$ is an {\em $\cL$-generating set for a cone} ($\cL$-GSC) 
if every vector in the intersection of the conic hull of the vectors in $S$ and 
$\cL^n$ can be expressed as a conic combination of the vectors in $S$ with multipliers in~$\cL$ \cite{abdi23b}.
Here is our key result,
\begin{theorem}\label{main-char}
Let $M\in\Z^{m\times n}, b\in\Z^m$ and let $\cL$ be a heavy set.
Denote by $M^=x=b^=$ the implicit equalities of $Mx\leq b$.
Then, $Mx\leq b$ is TD in $\cL$ if and only if both of the following conditions hold: 
\begin{enumerate}[\;i.]
\item 
The rows of $M^=$ form an $\cL$-GSC;
\item 
For every admissible $w\in\cL^n$, denote by $F$ the set of optimal solutions of $(P\!:\!M,b,w)$ and let $F^+$ be a down-face of $F$.
Then, the $(w,F,F^+)$-tilt constraint has a solution with variables in~$\cL$.
\end{enumerate}
\end{theorem}
\noindent
The proof is postponed until \S\ref{sec-gen-proof}.
Note that condition (i) is vacuously true when the polyhedron $\{x: Mx\leq b\}$ is full-dimensional.
We have the following analogue of \Cref{restricted-w} for tilt constraints.
\begin{RE}\label{restricted-w2}
Let $\cL:=\cL(\cS)$ where $\cS$ is a set of primes.
It suffices to consider in condition (ii) of \Cref{main-char}, $(w,F,F^+)$-tilt constraints for $w\in\Z^n$.
\end{RE}
\begin{proof}
Consider a $(w,F,F^+)$-tilt constraint with $w\in\cL^n$.
For some $\mu$ that is a product of integer powers of primes in $\cS$ we have $\mu w\in\Z^n$.
If $\bar{u}$ is a solution for the $(\mu w,F,F^+)$-tilt constraint with entries in $\cL$, 
then $\frac1\mu\bar{u}$ is a solution for the $(w,F,F^+)$-tilt constraint with entries in $\cL$.
\end{proof}
\noindent We are now ready for our first application of \Cref{main-char}.
\begin{proof}[Proof of \Cref{pq-to-L}]
Consider an arbitrary $(w,F,F^+)$-tilt constraint $\alpha^\top u=1$ of $Mx\leq b$ where $w\in\Z^n$.
Because of \Cref{main-char} and \Cref{restricted-w2}, it suffices to show that $\alpha^\top u=1$ has a solution in $\cL_q$ for every prime $q$.
To that end, we will show that $\alpha^\top u=1$ has an integral solution.
Consider $i\in[k]$.
\Cref{main-char} implies that  $\alpha^\top u=1$ has a solution $u^i$ with entries in $\cL(\cS_i)$.
For some $\mu_i$ that is a product of integer powers of primes in $\cS_i$ we have $\mu_iu^i$ is integral.
Since the gcd of $\mu_1,\ldots,\mu_k$ equals $1$, by Bezout's lemma we have $s_1,\ldots,s_k\in\Z$ for which $\sum_{i\in[k]}s_i\mu_i=1$.
Let $\bar{u}=\sum_{i\in[k]}s_i\mu_iu^i$. Then
\[
\alpha^\top\bar{u}=\sum_{i\in[k]}s_i\mu_i\alpha^\top u^i=\sum_{i\in[k]}s_i\mu_i=1.
\]
Thus $\bar{u}$ is an integral solution to $\alpha^\top u=1$ as required.
\end{proof}
\noindent
The reader might have noticed that in the proof \Cref{pq-to-L} it is shown that if an affine space contains points in $\cL(\cS_i)$ for all $i\in[k]$, then it also contains an integer point. 
We can apply that same idea to the affine hull of the optimal solutions to $(D\!:\!M,b,w)$ and use \Cref{density} for an alternate proof.
\section{Applications of \Cref{main-char}}\label{sec-applications}
In this section, we show how \Cref{main-char} yields sufficient conditions for a system to be Totally Dual in $\cL$ for various choices of $\cL$.
First we need to define \emph{braces}.
\subsection{From braces to solutions of tilt constraints}
Consider $M\in\Z^{m\times n}$ and $b\in\Z^m$. Suppose that $Mx\leq b$ has faces $F,F^+$ where $F^+$ is a down-face of $F$.
We say that a pair $(\ii,\rho)$ is an {\em $(F,F^+)$-brace} for $Mx\leq b$ if the following conditions hold,
\begin{enumerate}[\;b1.]
\item $\rho\in\Z^n$ and $\rho\in\aff(F^+)\setminus\aff(F)$,
\item $\ii\in I_{M,b}(F)\setminus I_{M,b}(F^+)$, 
\item $|b_{\ii}-\row_{\ii}(M)\rho|>0$.
\end{enumerate}
Informally, $(\ii,\rho)$ certifies that $F$ and $F^+$ are distinct faces, namely, $\rho$ is in the affine hull of $F^+$,
but $\rho$ is not in the affine hull of $F$ as it does not satisfy the implicit equality $\row_{\ii}(M)x=b_{\ii}$ of $F$.
The {\em gap} of brace $(\ii,\rho)$ is defined as $|b_{\ii}-\row_{\ii}(M)\rho|$. 
Note, that the gap of a brace is a positive integer. 

\Cref{main-char} highlights the importance of solutions to tilt constraints.
We will see that braces provide a means to finding such solutions.
Before we proceed, we require a definition.
Let $p$ be a positive integer. 
We say that a set $\cL\subseteq\R$ is {\em closed under $p$-division} if for every $\alpha\in\cL$ we have $\frac1p\alpha\in\cL$.
For instance $p$-adic rationals are closed under $p$-division.
\begin{PR}\label{use-brace-for-tilt}
Let $M\in\Z^{m \times n}$, $b\in\Z^m$ and assume that $\{x \in \R^n : Mx\leq b\}$ is an integral polyhedron.
Let $\cL$ be a heavy set that is closed under $p$-division for some positive integer $p$.
Consider $w\in\cL^n$ that is admissible for $M,b,w$.
Let $F$ be the optimal face of $(P\!:\!M,b,w)$ and let $F^+$ be a down-face of $F$.
If there exists an $(F,F^+)$-brace with gap $p$ then the $(w,F,F^+)$-tilt constraint has a solution with all variables in $\cL$.
\end{PR}
\begin{proof}
Denote by $(\ii,\rho)$ the $(F,F^+)$-brace. 
Let $I:=I_{M,b}(F)\setminus I_{M,b}(F^+)$.
By \Cref{ind-rep} the $(w,F,F^+)$-tilt constraint is independent of the choice of $\rho\in\aff(F^+)\setminus\aff(F)$.
Thus, we may assume it is of the form,
\begin{equation}\label{tilt-eq-local}
\frac1{\tau-w^\top\rho}\sum_{i\in I}[b_i-\row_i(M)\rho]u_i=1.
\end{equation}
Set $u_i:=0$ for all $i\in I$ where $i\neq\ii$ and set,
\begin{equation}\label{single-term}
u_{\ii}:=\frac{\tau-w^\top\rho}{b_{\ii}-\row_{\ii}(M)\rho}.
\end{equation}
Then $u$ is a solution to \eqref{tilt-eq-local}.
To complete the proof, it suffices to show that $u_{\ii}\in\cL$.
Since $\{x \in \R^n : Mx\leq b\}$ is integral, $F$ contains a point $\bar{x}\in\Z^n$.
Recall that $\tau$ is the optimal value of $(P\!:\!M,b,w)$, thus we have $\tau=w^\top\bar{x}$.
Since $w\in\cL^n$ and $\bar{x}\in\Z^n$ we have $\tau\in\cL$ (we are using the fact that $(\cL,+)$ is a subgroup).
Since $\rho\in\Z^n$ we also have $w^\top\rho\in\cL$.
Hence, the numerator in \eqref{single-term} is in $\cL$.
The denominator is the gap of the $(F,F^+)$-brace, i.e. $p$.
Since, $\cL$ is closed under $p$-division, $u_i\in\cL$ as required.
\end{proof}
\subsection{From resiliency to braces}
Consider matrix $M\in\Z^{m\times n}$ and vector $b\in\Z^m$. Let $p$ be a positive integer.
The system $Mx\leq b$ is said to be {\em $\frac1p$-resilient} if 
(i) $Q:=\{x \in \R^n : Mx\leq b\}$ is an integral polyhedron and
(ii) for every $i\in[m]$,
\[
Q\cap\{x \in \R^n : \row_i(M)x\leq b_i-s(i)\}
\]
is also an integral polyhedron for some $s(i)\in[p]$.
Hence, $Mx\leq b$ is resilient iff it is $1$-resilient and it is half-resilient iff it is $\frac12$-resilient.

Next, we use the notion of $\frac1p$-resiliency to obtain braces with gap $p$.
\begin{PR}\label{resilient-to-brace}
Consider $M\in\Z^{m\times n}$ and $b\in\Z^m$.
Let $F$ be a proper face of $Mx\leq b$ and let $F^+$ be a down-face of $F$.
If $Mx\leq b$ is $\frac1p$-resilient, then there exists an $(F,F^+)$-brace with gap at most $p$.
\end{PR}
\begin{proof}
Since faces of a polyhedron are obtained by setting suitable inequalities to equality, it follows readily that,
if $Mx\leq b$ is $\frac1p$-resilient for some positive integer $p$, then so is every face of $Mx\leq b$.
We may thus assume that $F^+=\{x \in \R^n : Mx\leq b\}$.
Pick $\ii\in I:=I_{M,b}(F)\setminus I_{M,b}(F^+)$.
Define the following optimization problem,
\begin{equation}\label{opt-gap}
\kappa := \sup\{b_{\ii}-\row_{\ii}(M)x : x\in F^+\}.
\end{equation}
By the choice of $\ii$, observe that $\kappa>0$.
Consider first the case where $\kappa\leq p$. 
Since $Mx\leq b$ is $\frac1p$-resilient, $F^+$ is integral. Since $\kappa$ is finite,
it follows that \eqref{opt-gap} attains its maximum at a point $\rho\in F^+ \cap \Z^n$
(since $F^+$ is integral, every nonempty face of $F^+$ contains
an integral vector).
However, then $(\ii,\rho)$ is an $(F,F^+)$-brace with gap $\kappa\leq p$ as required.
Thus, we may assume that $\kappa>p$.
Since $Mx\leq b$ is $\frac1p$-resilient there exists integer $s(\ii)\in[p]$ for which $Q:=F^+\cap\{x \in \R^n : \row_{\ii}(M)x\leq b_{\ii}-s(\ii)\}$ is an integral polyhedron.
As $\kappa>p$ it follows that the face $Q' := Q\cap\{x \in \R^n : \row_{\ii}(M)x=b_{\ii}-s(\ii)\}$ of $Q$ is non-empty.
Since $Q$ is integral, so is $Q'$.
Hence, $Q'$ contains an integral point $\rho$.
However, then again $(\ii,\rho)$ is an $(F,F^+)$-brace with gap $s(\ii) \leq p$ as required.
\end{proof}
\subsection{Totally Dual in $\cL$ - sufficient conditions}
The next result gives sufficient conditions for $Mx\leq b$ to be TD in $\cL$.
\begin{theorem}\label{key-sufficient}
Let $M\in\Z^{m\times n}$ and $b\in\Z^m$ and let $\cL\subseteq\R$ where $\cL$ is a heavy set.
Let $p$ denote a positive integer, and assume that $\cL$ is closed under $q$-divisions for all $q\in\{2,\ldots,p\}$.
Denote by $M^=x= b^=$ the implicit equalities of $Mx\leq b$.
Then $Mx\leq b$ is TD in $\cL$, if both of the following conditions hold,
\begin{enumerate}[\;i.]
\item The rows of $M^=$ form an $\cL$-GSC.
\item $Mx\leq b$ is $\frac1p$-resilient.
\end{enumerate}
\end{theorem}
\begin{proof}
Consider an admissible $w\in\cL^n$, denote by $F$ the set of optimal solutions of $(P\!:\!M,b,w)$ and let $F^+$ be a down-face of $F$.
Because of \Cref{main-char}, it suffices to show that the $(w,F,F^+)$-tilt constraint has a solution with variables in~$\cL$.
Since $Mx\leq b$ is $\frac1p$-resilient, there exists by \Cref{resilient-to-brace} an $(F,F^+)$-brace $(\ii,\rho)$ with gap $q\leq p$.
By hypothesis, $\cL$ is closed under $q$-division, and $\{x \in \R^n : Mx\leq b\}$ is an integral polyhedron, as $Mx\leq b$ is $\frac1p$-resilient.
It then follows from \Cref{use-brace-for-tilt} that the $(w,F,F^+)$-tilt constraint has a solution in $\cL$ as required.
\end{proof}
\noindent
Given a prime $p$, let $p_1,\ldots,p_{\ell}:=p$ denote the set of all primes up to and including $p$.
We write $\cL_{[p]}$ for $\cL(\{p_1,\ldots,p_{\ell}\})$.
In particular, $\cL_{[p]}$ is a heavy set.
We have the following immediate consequences of \Cref{key-sufficient},
\begin{CO}\label{sufficient-res-co}
Let $M\in\Z^{m\times n}$ and $b\in\Z^m$ and suppose that $\{x \in \R^n : Mx\leq b\}$ is full-dimensional. 
\begin{enumerate}[\;\;(a)]
\item If $Mx\leq b$ is resilient, then it is TD in $\cL$ for every heavy set $\cL$.
\item If $Mx\leq b$ is half-resilient, then it is TDD.
\item If $Mx\leq b$ is $\frac1p$-resilient for an integer $p\geq 2$, then it is TD in $\cL_{[p]}$.
\end{enumerate}
\end{CO}
\begin{proof}
Hypothesis (i) in \Cref{key-sufficient} trivially holds as $\{x \in \R^n : Mx\leq b\}$ is full-dimensional.
Applying \Cref{key-sufficient} for the case where:
$p=1$ yields (a);  
$p=2$ yields (b) as $\cL_2$ is closed under $2$-divisions; and
$p\geq 3$ yields (c) as $\cL_{[p]}$ is closed under $q$-divisions for every $q \in \{2,\ldots,p\}$.
\end{proof}
We say that $Mx\leq b$ is {\em $p$-small} if 
(i) $\{x \in \R^n : Mx\leq b\}$ is an integral polytope and 
(ii) for every $i\in[m]$ and every extreme point $\bar{x}\in Q$ we have $\row_i(M)\bar{x}\geq b_i-p$.
Equivalently (ii) holds when the largest entry of the {\em slack matrix}\footnote{see \cite[\S 4.10]{CCZ2014} for a definition.} of $Mx\leq b$ is at most $p$.
\begin{RE}\label{small-to-resilient}
If $Mx\leq b$ is $p$-small, then $Mx\leq b$ is $\frac1p$-resilient.
\end{RE}
\begin{proof}
Suppose $Mx\leq b$ is $p$-small. Pick $i\in[m]$ where $\row_i(M)x=b_i$ is not an implicit equality of $\{x \in \R^n : Mx\leq b\}$. The optimal solution to,
\[
s(i):=\max\{b_i-\row_i(M)x:Mx\leq b\},
\]
is attained by an extreme point of $Q=\{x:Mx\leq b\}$ and $s(i)\leq p$.
Then $Q\cap \{x \in \R^n :  \row_i(M) x \leq b_i-s(i)\}$ is a face of $Q$ and hence is integral.
By definition, this implies that $Mx\leq b$ is $\frac1p$-resilient.
\end{proof}
\noindent
Combining this observation with \Cref{sufficient-res-co} we deduce that for $M\in\Z^{m\times n}, b\in\Z^m$,
and $\{x:Mx\leq b\}$ full-dimensional, when $Mx\leq b$ is $2$-small it is TDD.
\section{Total Dual Integrality}\label{sec-tdi}
In this section, we give a geometric characterization of non-degenerate TDI systems.
Note that a $\Z$-GSC is also known as a {\em Hilbert Cone}~\cite{GP79}.
We are ready to state the main result of this section.
\begin{theorem}\label{gen-tdi}
Let $M\in\Z^{m\times n}$ and $b\in\Z^m$ where $Mx\leq b$ is non-degenerate.
Let $M^=x=b^=$ denote the implicit equalities of $Mx\leq b$.
Then $Mx\leq b$ is TDI if and only if both the following conditions hold:
\begin{enumerate}[\;i.]
\item The rows of $M^=$ form a Hilbert Cone;
\item $Mx\leq b$ is resilient.
\end{enumerate}
\end{theorem}
\noindent
Observe that when $\{x \in \R^n : Mx\leq b\}$ is full-dimensional there are no implicit constraints and (i) holds trivially.
It follows that \Cref{intro-tdi} is an immediate consequence of \Cref{gen-tdi}.
In \S\ref{sec-suf} we show that (i) and (ii) imply that $Mx\leq b$ is TDI.
Finally, in \S\ref{sec-nec} we show that $Mx\leq b$ TDI implies that (i) and (ii) hold.
\subsection{Sufficiency for TDI} \label{sec-suf}
Restricting \Cref{key-sufficient} to the case $p=1$ yields,
\begin{CO}\label{key-sufficient-p1}
Let $M\in\Z^{m\times n}$ and $b\in\Z^m$ and let $\cL\subseteq\R$ where $\cL$ is a heavy set.
Let $M^=x= b^=$ denote the implicit equalities of $Mx\leq b$.
Then $Mx\leq b$ is TD in $\cL$ if the following conditions hold,
\begin{enumerate}[\;i.]
\item The rows of $M^=$ form an $\cL$-GSC;
\item $Mx\leq b$ is resilient.
\end{enumerate}
\end{CO}
\noindent
We can now prove the promised result.
\begin{PR}\label{pr-suf}
In \Cref{gen-tdi} conditions (i) and (ii) imply $Mx\leq b$ is TDI.
\end{PR}
\begin{proof}
Let $p$ be an arbitrary prime and pick $w\in\cL_p^n$ that is a conic combination of the rows in $M^=$.
Then for some integer $k\geq 0$, $p^k w\in\Z^n$.
It follows by (i) in \Cref{gen-tdi} that $p^k w$ can be expressed as an integer conic combination of the rows in $M^=$.
Hence, $w$ can be expressed as a $p$-adic conic combination of the rows in $M^=$.
In particular, (i) holds in \Cref{key-sufficient-p1} for $\cL=\cL_p$.
We can thus apply \Cref{key-sufficient-p1} and deduce that $Mx\leq b$ TD in $\cL_p$.
As $p$ was an arbitrary prime, $Mx\leq b$ is near-TDI and the result follows from \Cref{non-deg-near}.
\end{proof}
\subsection{Necessity for TDI} \label{sec-nec}
A system $Mx\leq b$ is TDI if for every face $F$, the left hand side of the tight constraints for $F$ form a Hilbert Cone~\cite{GP79}.
It follows that condition (i) in \Cref{gen-tdi} is necessary for $Mx\leq b$ to be TDI.
The goal of this section is to prove the next result (thereby completing the proof of \Cref{gen-tdi}).
\begin{PR}\label{pr-nec}
In \Cref{gen-tdi} $Mx\leq b$ TDI implies $Mx\leq b$ is resilient.
\end{PR}
\noindent
To that end, we first require a definition.
Let $M\in\Z^{m\times n}$ and let $b,b'\in\Z^m$.
Let $F$ and $F'$ be faces of $Mx\leq b$ and of $Mx\leq b'$, respectively.
We say that $F$ and $F'$ are {\em mates} if $I_{M,b}(F)=I_{M,b'}(F')$, i.e. the index sets of the tight constraints for $F$ and $F'$ are the same.
We denote by $e^i$ the vector with $e^i_i=1$ and $e^i_j=0$ for all $j\neq i$.
For a polyhedron $Q := \{x \in \R^n : Mx\leq b\}$ the {\em lineality space} of $Q$ is denoted, $\lin(Q)$.
\begin{PR}\label{mates}
Let $M\in\Z^{m\times n}$ and let $b,b'\in\Z^m$ where $b'=b-e^i$ for some $i\in[m]$.
Suppose that $P:= \{x \in \R^n : Mx\leq b\}$ is integral and $Mx\leq b$ is non-degenerate.
Then, every minimal face $F'$ of $Mx\leq b'$ that is not a face of $Mx\leq b$ has a mate $F$ of $Mx\leq b$.
\end{PR}
\begin{proof}
Let $P'=\{x \in \R^n : Mx\leq b'\}$. Note that $\lin(P)=\{x \in \R^n : Mx=\0\}=\lin(P')$.
There exist pointed polyhedra $R$ and $R'$ with the property that $P=R+\lin(P)$ and $P'=R'+\lin(P')$.
Thus, minimal faces of $P$ and $P'$ have the same dimension $d:=\dim(\{x \in \R^n : Mx=\0\})$.
Let $\alpha:=\row_i(M)$ and let $\beta:=b_i$. Note, that $\alpha\in\Z^n$ and $\beta\in\Z$.
\begin{claim*}
$S:=\{x \in \R^n : \beta-1<\alpha^{\top} x<\beta\}\cap \Z^n=\emptyset$.
\end{claim*}
\begin{cproof}
If $\bar{x}\in S \cap \Z^n$ then $\alpha^\top \bar{x}\in\Z$, but $\beta-1<\alpha^{\top}\bar{x}<\beta$, a contradiction.
\end{cproof}
\noindent
Let $F'$ be a minimal face of $P'$ that is not a face of $P$. We will show that $F'$ has a mate $F$ of $Mx\leq b$.
Since $F'$ is not a face of $P$, it is the intersection of a face $L$ of dimension $d+1$ of $P$ and the hyperplane $\{x \in \R^n : \alpha^{\top} x=\beta-1\}$.
Hence, $L\cap\{x \in \R^n : \alpha^{\top} x<\beta-1\}$ and $L\cap\{x \in \R^n : \alpha^{\top} x>\beta-1\}$ are non-empty.
$L$ corresponds to a face $L_q$ of dimension $1$ of $R$ where $L=L_q+\lin(P)$.
Then $L_q\cap\{x \in \R^n : \alpha^{\top} x<\beta-1\}$ and $L_q\cap\{x \in \R^n : \alpha^{\top} x>\beta-1\}$ are non-empty.
Since $\alpha^{\top} x\leq\beta$ is not a redundant constraint of $Mx\leq b$, it follows that $L_q$ is a line segment or a half-line with an end $v_q$ where $\alpha^{\top} v_q>\beta-1$.
Then $v_q$ is an extreme point of $R$, hence $F:=\{v_q\}+\lin(P)$ is a minimal face of $P$.
Since $P$ is integral $F$ contains an integral point $v$ in $G$. 
It follows from the Claim that $\alpha^{\top} v=\beta$.
Thus, $F$ is defined by the tight constraints of $L$ and $\alpha^{\top} x=\beta$.
Since $F'$ is defined by the  tight constraints of $L$ and $\alpha^{\top} x=\beta-1$, $F$ is the mate of $F'$.
\end{proof}
\noindent
Consider a set $S=\{a^1,\ldots,a^m\}$ of vectors in $\Z^n$.
$S$ is an {\em $\cL$-generating set for a subspace} ($\cL$-GSS) 
if every vector in the intersection of the linear hull of the vectors in $S$ and $\cL^n$ can be expressed as a linear combination of the vectors in $S$ with multipliers in~$\cL$.
\begin{proof}[\Cref{pr-nec}]
Since $b$ is integral and $Mx\leq b$ is TDI, $\{x \in \R^n : Mx\leq b\}$ is an integral polyhedron.
Now pick $i\in[m]$ and let $b' := b-e^i$.
We need to show that $\{x \in \R^n : Mx\leq b'\}$ is an integral polyhedron.
Consider an arbitrary minimal face $F'$.
By \Cref{mates}, we may assume $\{x \in \R^n : Mx\leq b\}$ has a minimal face $F$ that is a mate of $F'$.
Denote by $Nx=f$ the tight constraints of $Mx\leq b$ for $F$.
Since $Mx\leq b$ is TDI, the rows of $N$ form a Hilbert basis~\cite{GP79}, or equivalently, a $\Z$-GSC.
It is proved in \cite{abdi23b}, that every $\Z$-GSC is an $\Z$-GSS and that if the rows of a matrix form a $\Z$-GSS, then so do the columns
(see \Cref{GSC-GSS} and \Cref{GSS} in \S\ref{sec-gen-proof} for details).
It follows that the columns of $N$ form a $\Z$-GSS.
Minimal faces are affine spaces.
Thus $F'=\{x \in \R^n : Nx=f'\}$ where $Nx=f'$ denotes the tight constraints for $F'$ of $Mx\leq b'$.
Since $b'\in\Z^m$, and since the columns of $N$ form a $\Z$-GSS, there exists an integral solution $x'$ to $Nx=f'$.
Then $x'$ is an integral point of $F'$. 
We proved that every minimal face of $\{x \in \R^n : Mx\leq b'\}$ has an integral point, i.e., that it is an integral polyhedron as required.
\end{proof}
\section{The Dyadic Conjecture}\label{sec-seymour}
Define $\zP:=\{0\}\cup\{2^i:i \in \Z_+\}$. We prove,
\begin{theorem}\label{main-inter-n}
Let $\zC$ be an ideal clutter. If for all $S\in\zC$ and $B\in b(\zC)$ we have $|S\cap B|-1\in\zP$ then $T(\zC)x\geq\1,x\geq\0$ is TDD.
\end{theorem}
\noindent
A clutter $\zC$ is said to be {\em binary} if for every $S\in\zC$ and every $B\in b(\zC)$ we have $|S\cap B|$ odd.
See \cite{seymour1987} for a characterization. Observe that \Cref{main-inter-n} implies,
\begin{CO}\label{dyadic-co}
Let $\zC$ be an ideal clutter. 
Then $T(\zC)x\geq\1,x\geq\0$ is TDD if either
(a) for all $S\in\zC$ and $B\in b(\zC)$, we have $|S\cap B|\leq 3$, or
(b) $\zC$ is a binary clutter and for all $S\in\zC$ and $B\in b(\zC)$, we have $|S\cap B|\leq 5$.
\end{CO}
The proof of \Cref{main-inter-n} relies on \Cref{main-char}.
We therefore need to restate the set covering linear program and its dual in the setting of the linear programs $(P\!:\!M,b,w)$ and $(D\!:\!M,b,w)$.
To that effect, we assume when we consider a clutter $\zC$ that the number of members of $\zC$ is $m$ and that the size of the ground set is $n$.
Recall that for a clutter $\zC$, $T(\zC)$ denotes the $0,1$ matrix with rows corresponding to characteristic vectors of $\zC$. 
Thus $T(\zC)$ is an $m\times n$ matrix. Define the matrix,
\[
M(\zC):=\left(\begin{array}{c} -T(\zC) \\ -I_n\end{array}\right),
\]
where $I_n$ denotes the $n\times n$ identity matrix.
In addition, we define vector $d(\zC)\in\{-1,0\}^{m+n}$ where $d(\zC)_i=-1$ for $i=1,\ldots,m$ and $d(\zC)_i=0$ for $i=m+1,\ldots,n$.
Then the primal-dual pair of set covering linear programs,
\begin{align*}
& \min\{w^\top x:T(\zC) x \geq\1,x\geq\0\} \\
& \max\{\1^\top y: T(\zC)^\top y\leq w,y\geq\0\},
\end{align*}
\noindent
can be expressed as $(P\!:\!M(\zC),d(\zC),-w)$ and $(D\!:\!M(\zC),d(\zC),-w)$ respectively.

We leave the following observation as an easy exercise.
\begin{RE}\label{TDD-reduction}
$T(\zC)x\geq\1,x\geq\0$ is TDD if and only if $M(\zC)x\leq d(\zC)$ is TDD.
\end{RE}
\noindent
Denote by $\rec(R)$ the recession cone of a polyhedron $R$.
A polyhedral cone $C\subseteq\R^n$ is {\em generated} by a set $S$ of vectors in $\R^n$ if $C$ is equal to the conic hull of the vectors in $S$.
We require the following observation.
\begin{RE}\label{easy-stuff}
Let $P$ be a polyhedron and suppose that $\rec(P)$ is generated by a finite set $S$ of vectors.
Let $P'$ be a face of $P$ then 
(a) $\rec(P')$ is a face of $\rec(P)$, and 
(b) $\rec(P')$ is generated by a subset $S'$ of $S$.
\end{RE}
\begin{proof}
Since $P$ is a polyhedron, then $P=\{x:Mx\leq b\}$ for some matrix $M$ and vector $b$.
{\bf (a)}
Let $I:=I_{M,b}(P')$, then $P'=P \cap\{x:\row_i(M)x=b_i, i\in I\}$.
We have $\rec(P)=\{x:Mx\leq\0\}$ \cite[Proposition 3.15]{CCZ2014}.
Then observe that $\rec(P')=\rec(P)\cap\{x:\row_i(M)x=0, i\in I\}$.
(b) 
$\rec(P')$ is generated by its extreme rays.
Observe that every extreme ray in the face of a polyhedral cone is an extreme ray in the polyhedral cone and use (a).
\end{proof}
We are now ready for the last proof of this section.
\begin{proof}[\Cref{main-inter-n}]
Recall that $\cL_2$ denotes the set of dyadic rationals.
Let $M:=M(\zC)$, $b:=d(\zC)$, $n=|E(\zC)|$ and $m= |\zC|$.
In light of \Cref{TDD-reduction} we need to prove that $Mx\leq b$ is TDD.
It suffices to prove that conditions (i) and (ii) in \Cref{main-char} hold.
Note, that (i) trivially holds as $\{x:Mx\leq b\}$ is full-dimensional.
Thus it suffices to prove (ii).
To that end, consider an admissible $w\in\cL_2^n$ of $(P\!:\!M,b,w)$, denote by $F$ the set of optimal solutions, and let $F^+$ be a down-face of $F$.
We need to show that the $(w,F,F^+)$-tilt constraint has a solution with variables in~$\cL_2$.
Because $\zC$ is ideal, $Q:=\{x:T(\zC)\geq\1,x\geq\0\}=\{x:Mx\leq b\}$ is integral. 
Observe that $\cL_2$ is closed under $2^k$ divisions for any $k\in\Z_+$.
Therefore, by \Cref{use-brace-for-tilt}, it suffices to show there exists an $(F,F^+)$-brace with gap $2^k$ for some $k\in\Z_+$.

\vspace{0.1in}\noindent{\bf Case 1.} $F^+\setminus F$ contains an extreme point $\rho$ of $Q$.

\vspace{0.1in}\noindent
Because $\zC$ is ideal, $\rho$ is the characteristic vector of a member $B$ of $b(\zC)$ \cite[Remark 1.16]{cornuejols2001}.
Let $\ii\in I_{M,b}(F)\setminus I_{M,b}(F^+)$.
Consider first the case where $\ii\in[m]$.
Constraint $\ii$ of $Mx\leq b$ says that $\sum_{i\in S}x_i\geq 1$ for some $S\in\zC$. Then
\[
b_{\ii}-\row_{\ii}(M)\rho=\sum_{i\in S}\rho_i - 1=|S\cap B|-1\in\zP, 
\]
where the membership follows by hypothesis. 
Hence, $(\ii,\rho)$ is the required $(F,F^+)$-brace and we may assume $\ii\in[m+n]\setminus[m]$.
Then constraint $\ii$ of $Mx\leq b$ is $x_{\ii-m}\geq 0$.
Therefore, there exists $\rho'\in F^+\setminus F$ with $\rho'_{\ii-m}>0$.
Since $\rec(Q)=\R^n_+$, it is generated by $S=\{e^1,\ldots,e^n\}$.
It follows from \Cref{easy-stuff} that $\rec(F^+)$ is generated by $S'\subseteq S$.
Therefore, $\rho'$ is obtained as the sum of a convex combination of extreme points of $F^+$ and a conic combination of $S'$.
As $Q$ is integral, so is $F^+$. 
It follows, that we must have $\rho''\in F^+\cap\Z^n$ with $\rho''_{\ii}=1$. Then 
\[
b_{\ii}-\row_{\ii}(M)\rho''=\rho''_{\ii-m}=1,
\]
and $(\ii,\rho'')$ is the required $(F,F^+)$-brace.

\vspace{0.15in}\noindent{\bf Case 2.} $F^+\setminus F$ contains no extreme point.

\vspace{0.1in}\noindent
Since we are in case 2 and $F^+\supset F$ we have $\rec(F^+)\supset\rec(F)$.
It follows from \Cref{easy-stuff} that there must be a generator $e^j$ in of $\rec(F^+)$ that is not a generator of $\rec(F)$.
Hence, for some integral point $v\in F$ we have $\rho:=v+e^j\in F^+\setminus F$.
Let $\ii$ denote the index of a constraint of $Mx\leq b$ that is tight for $v$ but not for $\rho$.
Then $(\ii,\rho)$ is again the required $(F,F^+)$-brace.
\end{proof}
\section{The proof of \Cref{main-char}}\label{sec-gen-proof}
In this section we give the proof of \Cref{main-char} (our characterization of systems of inequalities that are TD in $\cL$ for a heavy set $\cL$).
Along the way, we will also prove \Cref{ind-rep}.
This section is organized as follows:
in \S\ref{ax-density} we explain the role of density; 
\S\ref{ax-gen} reviews results on generating sets and cones;
we derive applications of complementary slackness in \S\ref{ax-affine};
a geometric interpretation of the tilt constraints and the proof of \Cref{ind-rep} is found in \S\ref{ax-geometry};
finally, in \S\ref{ax-proof} we present the proof of \Cref{main-char}.
\subsection{Density and a theorem of the alternative}\label{ax-density}
A key property of heavy sets is density. 
This allows us to use the following powerful result \cite{abdi23b,abdi23}.
\begin{theorem}\label{density}
Let $\cL$ be a heavy set. 
Consider a nonempty convex set $S$, where $\aff(S)$ is a translate of a rational subspace.
Then $S\cap\cL^n\neq\emptyset$ if and only if $\aff(S)\cap\cL^n\neq\emptyset$.
\end{theorem}
\noindent
Based on this result, it is shown in \cite{abdi23} that the problems of checking if a polyhedron contains a dyadic, or $p$-adic]] can be answered in polynomial time.
Contrast this with the NP-hardness of checking if a polyhedron contains an integral point (see, for instance, \cite[Chapter 2]{AroraBarak}). 

We require the following theorem of the alternative from \cite[Theorem 2.7]{abdi23}.
\begin{theorem}\label{alternative}
Suppose $(\cL,+)$ forms a subgroup of $(\R,+)$ and $\cL \neq \R$. 
Let $A\in\Z^{m\times n}$ and let $b\in\Z^m$. \\
Then exactly one of the following holds,
\begin{enumerate}[\;a.]
\item $Ax=b$ has a solution in $\cL^n$, 
\item there exists $u\in\R^m$ such that $A^\top u\in\Z^n$ and $b^\top u\notin\cL$.
\end{enumerate}
\end{theorem}
\subsection{Generating sets and cones}\label{ax-gen}
Recall the definition of $\cL$-GSC in \S\ref{sec-gsc-def} and of $\cL$-GSS in \S\ref{sec-nec}.
The following was proved for the case where $\cL$ is the set of dyadic rationals in \cite[Proposition~3.3]{abdi23b}.
The same proof extends to the following general context, we include it for completeness.
\begin{PR}\label{GSC-GSS} 
Suppose $(\cL,+)$ forms a subgroup of $(\R,+)$ and $\cL\supset\Z$.
If $\{a^1,\ldots,a^m\}\subset \Z^n$ is an $\cL$-GSC then it is an $\cL$-GSS.
\end{PR}
\noindent In \cite{abdi23b}, it was observed that the converse does not hold.
\begin{proof}[\Cref{GSC-GSS}]
Let $A\in \Z^{m\times n}$ be the matrix whose columns are $a^1,\ldots,a^n$. 
Take $b\in \cL^m$ such that $A\bar{x}=b$ for some $\bar{x} \in \R^n$. 
We need to show that the system $Ax=b$ has a solution in $\cL^n$.
To this end, let $\bar{x}':=\bar{x}-\lfloor \bar{x}\rfloor\geq \0$ and $b':=A\bar{x}'=b-A\lfloor \bar{x}\rfloor$.
Since $A\lfloor \bar{x}\rfloor\in\Z^m$, $\cL\supset\Z$, and
$(\cL,+)$ forms a group, $b'\in\cL^m$. 
By construction, $Ax=b',x\geq \0$ has a solution, namely $\bar{x}'$.
So it has a solution, say $\bar{z}'\in\cL^n$, as the columns of $A$ form an $\cL$-GSC. 
Let $\bar{z}:=\bar{z}'+\lfloor \bar{x}\rfloor$, which is also in $\cL^n$ since $\cL\supset\Z$ and
$(\cL,+)$ forms a group.
Then, $A\bar{z}=A\bar{z}'+A\lfloor \bar{x}\rfloor=b'+A\lfloor \bar{x}\rfloor=b$, so $\bar{z}\in\cL^n$ is a solution to $Ax=b$, as required.
\end{proof}
\noindent
The next result was proved for the case of $p$-adic rationals in \cite{abdi23b}.
The same proof also works for the case of integers. We include it for completeness.
\begin{PR}\label{GSS}
The following are equivalent for a matrix $A\in\Z^{m\times n}$,
\begin{enumerate}[\;a.]
\item the rows of $A$ form an $\Z$-GSS,
\item the columns of $A$ form an $\Z$-GSS,
\item whenever $u^\top A$ and $Ax$ are integral, then $u^\top Ax\in\Z$.
\end{enumerate}
\end{PR}
\begin{proof}
{\bf (b) $\Rightarrow$ (c)}
Choose $x$ and $u$ such that $u^\top A$ and $Ax$ are integral. Let $b=Ax\in\Z^n$. 
By (b), there exists $\bar{x}\in\Z^n$ such that $b=A\bar{x}$.
Thus, $u^\top Ax=u^\top A\bar{x}=(u^\top A)\bar{x}\in\Z$,
since $u^\top A\in\Z^m$ and $\bar{x}\in\Z^n$.
{\bf (c) $\Rightarrow$ (b)}
Pick $b\in\Z^m$ such that $Ax=b$ for some $x\in\R^n$.
We need to show that $Ax=b$ has a solution in $\Z$.
Pick any $u\in\R^m$ for which $u^\top A$ is integral.
Then by (c), $u^\top b=u^\top Ax\in\Z$. Since this holds for all $u$ such that
$u^{\top}A$ is integral,
it then follows from \Cref{alternative} that $Ax=b$ has a solution in $\Z$.
{\bf (a)  $\Leftrightarrow$ (c)}
Observe that (c) holds for $A$ if and only if it holds for $A^\top$.
Therefore, (b) $\Leftrightarrow$ (c) implies that (a) $\Leftrightarrow$ (c).
\end{proof}
\subsection{Affine hull}\label{ax-affine}
Recall that for a non-empty face $F$ of $\{x \in \R^n : Mx\leq b\}$ we have that $F=\{x \in \R^n : Mx\leq b\}\cap\{x \in \R^n : \row_i(M)x=b_i, i\in I_{M,b}(F)\}$ \cite[Theorem 3.24]{CCZ2014}.
The affine hull of a polyhedron is characterized by its implicit equalities \cite[Section 3.7]{CCZ2014}, namely,
\begin{PR}\label{face-hull}
Let $M^=x =  b^=$ be the implicit equalities of $Mx\leq b$. Then
\[\aff(\{x \in \R^n : Mx\leq b\})=\{x \in \R^n :  M^=x=b^=\}.\]
\end{PR}
\noindent
In order to use \Cref{density} for our purposes, we will need to get an explicit description of the affine hull of the optimal solutions to $(D\!:\!M,b,w)$.
First, let us state the Complementary Slackness conditions for the primal-dual pair $(P\!:\!M,b,w)$ and $(D\!:\!M,b,w)$.
\begin{equation}\label{CS-condition}
\mbox{For all row indices $i$ of $M$:}\; \row_i(M)x=b_i\;\mbox{or}\; y_i=0.
\end{equation}
\noindent
Recall for linear programming that a pair of primal, dual feasible solutions are both optimal if and only if complementary slackness holds.
We use complementary slackness to characterize the optimal solutions of $(P\!:\!M,b,w)$.
\begin{PR}\label{primal-opt-char}
A nonempty face $F$ of $\{x \in \R^n : Mx\leq b\}$ is contained in the optimal solutions of $(P\!:\!M,b,w)$ if and only if $w$ is a conic combination\footnote{We interpret $\cone(\emptyset):=\{\0\}$.} of $S:=\{\row_i(M): i\in I_{M,b}(F)\}$.
\end{PR}
\begin{proof}
Let $F$ be a nonempty face of $\{x \in \R^n : Mx\leq b\}$ and suppose that $w$ is a conic combination of $S$.
Then there exists $y\geq\0$ with $w=M^\top y$ and $y_i=0$ for all $i\notin I_{M,b}(F)$.
Then $y$ is feasible for $(D\!:\!M,b,w)$ and any $x\in F$ and $y$ satisfy \eqref{CS-condition}. 
Hence every $x\in F$ is an optimal solution of $(P\!:\!M,b,w)$.
Conversely, assume every $x \in F$ is an optimal solution of $(P\!:\!M,b,w)$.
Pick $x$ in the relative interior of $F$ and let $y$ be any optimal solution of $(D\!:\!M,b,w)$.
Then $y\geq\0$, $w=M^\top y$ by feasibility, and $y_i=0$ for all $i\notin I_{M,b}(F)$ by \eqref{CS-condition}.
It follows that $w$ is a conic combination of $S$.
\end{proof}
\noindent
Next, we use complementary slackness and strict complementarity to characterize the optimal solutions of $(D\!:\!M,b,w)$.
\begin{PR}\label{dual-opt-char}
Let $F$ be the set of optimal solutions to $(P\!:\!M,b,w)$ where $F\neq\emptyset$.
Let $y$ be a feasible solution to $(D\!:\!M,b,w)$. 
Then, $y$ is optimal if and only if $y_i=0$ for all $i\notin I_{M,b}(F)$.
Moreover, there exists an optimal solution $y$ with $y_i>0$ for all $i\in I_{M,b}(F)$.
\end{PR}
\begin{proof}
Suppose that $y$ is optimal.
Let $i\notin I_{M,b}(F)$, then by definition of $I_{M,b}(F)$, there exists $x^i\in F$ with $row_i(M)x^i<b_i$.
Since $x^i$ is optimal for $(P\!:\!M,b,w)$ it follows from \eqref{CS-condition} for the pair $x^i,y$ that $y_i=0$.
Suppose now that $y_i=0$ for all $i\notin I_{M,b}(F)$.
Pick any $\bar{x}\in F$, then $\bar{x},y$ satisfy \eqref{CS-condition}.
Hence, $y$ is optimal.
Finally, since $(P\!:\!M,b,w)$ has an optimal solution, by strict complementarity theorem, there exist primal and dual solutions $x$ and $y$, respectively which are strictly complementary. 
Then, $y_i>0$ for all $i\in I_{M,b}(F)$, as desired.
\end{proof}
\noindent
Next, we can characterize the affine hull of optimal solutions of $(D\!:\!M,b,w)$.
\begin{PR}\label{affine-hull-dual}
Let $F$ be the set of optimal solutions to $(P\!:\!M,b,w)$ and assume that $F\neq\emptyset$.
Then the affine hull of the optimal dual solutions of $(D\!:\!M,b,w)$ is given by,
\begin{equation}\label{affine-hull-eq}
\left\{y \in \R^m : M^\top y=w, y_i=0, i\notin I_{M,b}(F)\right\}.
\end{equation}
\end{PR}
\begin{proof}
By \Cref{dual-opt-char} the optimal solutions to $(D\!:\!M,b,w)$ are exactly the points in,
\[ 
Q=\left\{y \in \R^m: M^\top y=w, y\geq\0, y_i=0, i\notin I_{M,b}(F)\right\}. 
\]
By the "moreover" part of \Cref{dual-opt-char} $y_i\geq 0$ is not a tight constraint for $Q$ when $i\notin I_{M,b}(F)$.
The result now follows from \Cref{face-hull}.
\end{proof}
\subsection{A Geometric interpretation}\label{ax-geometry}
Next, we provide a geometric interpretation of the tilt constraints.
\begin{PR} \label{u-perturbation}
Consider $Mx\leq b$ with nonempty faces $F,F^+$ where $F^+$ is a down-face of $F$ and where $F$ is defined by a supporting hyperplane $\{x \in \R^n : w^\top x=\tau\}$.
Let $I:=I_{M,b}(F)\setminus I_{M,b}(F^+)\neq\emptyset$ and assume that for every $i\in I$ we are given $u_i\in\R$.
Define, 
\begin{equation}\label{u-perturbation-eq}
\bar{w}^{\top}:=w^{\top}-\sum_{i\in I}u_i\row_i(M)\quad\mbox{and}\quad\bar{\tau}:=\tau-\sum_{i\in I}u_ib_i.
\end{equation}
Then the following are equivalent,
\begin{enumerate}[\;a.]
\item 
$F^+\subseteq\{x \in \R^n : \bar{w}^\top x=\bar{\tau}\}$,
\item 
$\bar{w}\in\spn\{\row_i(M):i\in I_{M,b}(F^+)\}$,
\item 
the $(w,F,F^+)$-tilt constraint for $Mx\leq b$ is satisfied by $u_i: i\in I$.
\end{enumerate}
\end{PR}

\vspace{0.1in}\noindent{\bf \Cref{ex-part1} - continued.}
Recall, $M,b,w,\tau$ as in \eqref{ex-ridge}.
We had defined, faces $F=\{(0,3)\}$ and $F^+_1$ which consists of the line segment with ends $(0,3)$ and $(3,0)$.
Recall that $\rho_1=(3,0)$.
The $(w,F,F^+_1)$-tilt constraint has a unique solution $u_2=1$.
Let $\bar{w}=w-u_2\row_2(M)=(0,1)-(-1,0)=(1,1)$ and let $\bar{\tau}=\tau-u_2b_2=3$.
Then (c) holds in \Cref{u-perturbation} and the reader can verify that 
$F^+_1\subseteq\{x \in \R^n : \bar{w}^\top x=\bar{\tau}\}$ and that $\bar{w}^{\top}\in\spn\{\row_1(M)\}$.
\begin{proof}[\Cref{u-perturbation}] 
Let $\rho\in\aff(F^+)\setminus\aff(F)$ and $H=\{x \in \R^n : \bar{w}^\top x=\bar{\tau}\}$.
\begin{claim}\label{claim1-tilt}
$F\subseteq\{x \in \R^n : \bar{w}^\top x=\bar{\tau}\}$.
\end{claim}
\begin{cproof}
Let $x\in F$. 
Then $w^\top x=\tau$ and for all $i\in I\subseteq I_{M,b}(F)$, $\row_i(M) x=b_i$. Thus,
\[
\bar{w}^\top x= \left[w^{\top}-\sum_{i\in I}u_i\row_i(M)\right] x = \tau-\sum_{i\in I}u_ib_i=\bar{\tau},
\]
hence for every $x\in F$, $\bar{w}^\top x=\bar{\tau}$.
\end{cproof}
\begin{claim}\label{claim2-tilt}
$\bar{w}^\top \rho=\bar{\tau}$ if and only if (c) holds.
\end{claim}
\begin{cproof}
$\bar{\tau}=\bar{w}^\top\rho$ can be rewritten as,
\[
\tau-\sum_{i\in I}u_ib_i=w^\top\rho-\sum_{i\in I}u_i\row_i(M)\rho.
\]
Then, observe that this is the $(w,F^+,F)$-tilt constraint (with the terms rearranged).
\end{cproof}
\noindent \fbox{(a) $\Rightarrow$ (c)}
Since (a) holds, $\aff(F^+)\subseteq H$, hence $\bar{w}^\top\rho=\bar{\tau}$.
Then (c) holds by Claim~\ref{claim2-tilt}.
\noindent \fbox{(c) $\Rightarrow$ (a)}
By Claim~\ref{claim1-tilt} and Claim~\ref{claim2-tilt} we have $\aff(F\cup\{\rho\})\subseteq H$.
Since $\rho\in\aff(F^+)\setminus\aff(F)$ and $\dim(F^+)=\dim(F)+1$ we have,
$\aff(F\cup\{\rho\})=\aff(F^+)$.
Hence, $\aff(F^+)\subseteq H$ and (a) holds.
\fbox{(a) $\Rightarrow$ (b)}
Since $F^+\subseteq H$, we have $\aff(F^+)-\rho\subseteq H-\rho$ \footnote{Given $S\subseteq\R^n$ and $\rho\in\R^n$, $S-\rho:=\{s-\rho:s\in S\}$.}.
By \Cref{face-hull}, $\aff(F^+)=\{x \in \R^n : \row_i(M)x=b_i, i\in I_{M,b}(F^+)\}$.
Thus $\aff(F^+)-\rho=\{x \in \R^n : \row_i(M)x=0, i\in I_{M,b}(F^+)\}$ and similarly $H-\rho=\{x \in \R^n : \bar{w}^\top x=0\}$.
Therefore,
\begin{equation}\tag{$\star$}
\{x \in \R^n : \row_i(M)x=0, i \in I_{M,b}(F^+)\}\subseteq\{x \in \R^n : \bar{w}^\top x=0\}.
\end{equation}
Taking the orthogonal complement yields,
\begin{equation}\tag{$\dagger$}
\spn\{\row_i(A):i\in I_{M,b}(F^+)\}\supseteq \spn(\bar{w}^{\top}).
\end{equation}
Hence, (b) holds.
\fbox{(b) $\Rightarrow$ (a)}
If (b) holds then so does ($\dagger$) and in turn, so does ($\star$).
Equivalently, $\aff(F^+)-\rho\subseteq H-\rho$.
Hence, $\aff(F^+)\subseteq H$, and (a) holds.
\end{proof}
\noindent We close this section with the following observation,
\begin{RE}\label{ind-rep}
We get the same constraint in \eqref{tilt-eq} for every choice of $\rho\in\aff(F^+)\setminus\aff(F)$.
\end{RE}
\begin{proof}
Pick $\rho_1,\rho_2\in\aff(F^+)\setminus\aff(F)$ and for $j=1,2$ denote by (ej) the equation \eqref{tilt-eq} with $\rho=\rho_j$.
Suppose that $u$ is a solution to (e1) and let $\bar{w}^\top x=\bar{\tau}$ as defined in \eqref{u-perturbation-eq}.
By the equivalence between (a) and (c)  for (e1) in \Cref{u-perturbation}, $F^+\subseteq\{x \in \R^n : \bar{w}^\top x=\bar{\tau}\}$.
By the equivalence between (a) and (c)  for (e2), $u$ is a solution to (e2).
Therefore, (e1) and (e2) have the same set of solutions, and hence, (e1) and (e2) are related by scaling.
As the right-hand-sides of (e1) and (e2) both equal $1$, the result follows.
\end{proof}
\subsection{The proof of \Cref{main-char}}\label{ax-proof}
We require a number of preliminaries for the proof.
\begin{PR}\label{dual-satisfy-base}
For $w$ admissible, let $F$ denote the set of optimal solutions of $(P\!:\!M,b,w)$ and let $F^+$ be any down-face of $F$.
If $y$ is an optimal solution of $(D\!:\!M,b,w)$ then setting $u_i:=y_i$ for all $i\in I_{M,b}(F)\setminus I_{M,b}(F^+)$ satisfies the $(w,F,F^+)$-tilt constraint.
\end{PR}
\begin{proof}
Throughout the proof, we write $I$ for $I_{M,b}$.
For every $i\in I(F)\setminus I(F^+)$ set $u_i=y_i$.
By Complementary Slackness \eqref{CS-condition}, we have $y_i=0$ for all $i\notin I(F)$.
Since $y$ is feasible for $(D\!:\!M,b,w)$, $w=M^\top y$. Therefore,
\begin{equation}\tag{$\star$}
w=\sum_{i\in I(F)}\row_i(M)^\top y_i= \sum_{i\in I(F)\setminus I(F^+)}\row_i(M)^\top u_i+\sum_{i\in I(F^+)}\row_i(M)^\top y_i.
\end{equation}
Let $\bar{w}$ be as defined in \eqref{u-perturbation-eq}.
Then ($\star$) implies that $\bar{w}=\sum_{i\in I(F^+)}\row_i(M)^\top y_i$.
The result now follows by the equivalence of (b) and (c) of \Cref{u-perturbation}.
\end{proof}
Consider $M\in\Z^{m\times n}$, $b\in\Z^m$ and let $\cL\subset\R$.
We say that $w\in\cL^n$ is {\em $\cL$-bad} with respect to $Mx\leq b$, if
(i) $w$ is admissible for $(P\!:\!M,b,w)$ and 
(ii) $(D\!:\!M,b,w)$ has no optimal solution in $\cL^m$.
By definition, $Mx\leq b$ is TD in $\cL$ if and only if there is no $w \in \cL^n$ that is $\cL$-bad.
A vector $w\in\cL^n$ is {\em $\cL$-extremal} with respect to $Mx\leq b$, if
it is $\cL$-bad and among all $\cL$-bad vectors it maximizes the dimension of the set of optimal solutions of $(P\!:\!M,b,w)$.
We have the following relation between extremal weights and the tilt-constraints.
\begin{PR}\label{extremal-key}
Let $M\in\Z^{m\times n}, b\in\Z^m$ and let $\cL$ be a heavy set.
Suppose that $w$ is $\cL$-extremal with respect to $Mx\leq b$.
Let $F$ denote the set of optimal solutions of $(P\!:\!M,b,w)$ and let $F^+$ be a down-face of $F$.
Then the $(w,F,F^+)$-tilt constraint has no solution with all variables in $\cL$.
\end{PR}
\begin{proof}
We write $I$ for $I_{M,b}$ and let $S:=\{\row_i(M) : i\in I(F^+)\}$.
\begin{claim}
The set $S$ forms a $\cL$-GSS.
\end{claim}
\begin{cproof}
Let $\Omega\in\cL^n$ be an arbitrary vector in the conic hull of $S$.
Denote by $F^{\Omega}$ the optimal face of $(P\!:\!M,b,\Omega)$.
By \Cref{primal-opt-char}, $F^+\subseteq F^\Omega$.
Since $\dim(F^\Omega)\geq\dim(F^+)>\dim(F)$ and since $w$ is extremal, $(D\!:\! M,b,\Omega)$ has an optimal solution $z\in\cL^m$.
By \eqref{CS-condition}, $z_i=0$ for all $i\notin I(F^\Omega)$.
As $F^+\subseteq F^\Omega$, $I(F^+)\supseteq I(F^\Omega)$, hence, $z_i=0$ for all $i\notin I(F^+)$.
Hence, 
\[
\Omega=M^\top z=\sum_{i\in I(F^+)}\row_i(M)z_i,
\]
and $\Omega$ is a conic combination of vectors in $S$ with coefficients in $\cL$.
Since $\Omega$ was arbitrary, $S$ forms an $\cL$-GSC.
The result then follows from \Cref{GSC-GSS}.
\end{cproof}
\noindent
Suppose for a contradiction that the $(w,F,F^+)$-tilt constraint has a solution with $u_i\in\cL$ for all $i\in I(F)\setminus I(F^+)$.
Define,
\begin{equation}\label{extremal-eq1}
\bar{w}^{\top}=w^{\top}-\sum_{i\in I(F)\setminus I(F^+)}u_i\row_i(M).
\end{equation}
By the equivalence between (b) and (c) of \Cref{u-perturbation}, $\bar{w}$ is in the span of $S$.
It then follows from the claim that there exists a vector $z$ where $z_i\in\cL$ for all $i\in I(F^+)$ for which,
\begin{equation}\label{extremal-eq2}
\bar{w}^{\top}=\sum_{i\in I(F^+)}z_i\row_i(M).
\end{equation}
Define, $\bar{y}$ where for every $i\in[m]$, 
\[
\bar{y}_i :=
\begin{cases}
z_i & \mbox{if $i\in I(F^+)$} \\
u_i & \mbox{if $i\in I(F)\setminus I(F^+)$} \\
0 & \mbox{otherwise.}
\end{cases}
\]
By \eqref{extremal-eq1} and \eqref{extremal-eq2} we have $M^\top\bar{y}=w$.
Since in addition, $\bar{y}_i=0$ for all $i\notin I(F)$, 
\Cref{affine-hull-dual} implies that $\bar{y}$ is in the affine hull of optimal solutions to $(D\!:\!M,b,w)$.
By construction, $\bar{y}\in\cL^m$ and \Cref{density} implies 
that there is an optimal solution to $(D\!:\!M,b,w)$ in $\cL^m$.
Hence, $w$ is not $\cL$-bad, a contradiction, as $w$ is $\cL$-extremal.
\end{proof}
\noindent We are now ready for the main proof in this section.
\begin{proof}[Proof of \Cref{main-char}]
Throughout this proof we write $I$ for $I_{M,b}$.
We first assume that $Mx\leq b$ is TD in $\cL$ and will show that both (i) and (ii) hold.
Consider first (i).
Let $\Omega\in\cL^n$ be a conic combination of the rows of $M^=$.
It follows from \Cref{primal-opt-char} that every feasible solution of $(P\!:\!M,b,\Omega)$ is an optimal solution.
Pick $x$ in the relative interior of $\{x \in \R^n : Mx\leq b\}$.
Since $Mx\leq b$ is TD in $\cL$ there exists an optimal solution $y\in\cL^m$ of $(D\!:\!M,b,\Omega)$.
By \eqref{CS-condition}, $y_i=0$ for all $i$ that does not correspond to a row of $M^=$.
Since $y\geq\0$ and $\Omega=M^\top y$, $\Omega$ is a conic combination of $M^=$ with coefficients in $\cL$.
As $\Omega$ was arbitrary the rows of $M^=$ forms a GSC in $\cL$.
Consider (ii).
Since $Mx\leq b$ is TD in $\cL$, there exists an optimal solution $y\in\cL^m$ of $(D\!:\!M,b,w)$.
Then by \Cref{dual-satisfy-base}, $u_i=y_i$ for all $i\in I(F)\setminus I(F^+)$ satisfies the $(w,F,F^+)$-tilt constraint.

Assume now that (i) and (ii) hold and suppose for a contradiction that $Mx\leq b$ is not TD in $\cL$.
Then there exists $w$ that is $\cL$-extremal. 
Denote by $F$ the optimal face for $(P\!:\!M,b,w)$.
Consider first the case where $F=\{x \in \R^n : Mx\leq b\}$.
\Cref{primal-opt-char} implies that $w$ is a conic combinations of the rows of $M^=$.
Since the rows of $M^=$ form an $\cL$-GSC, there exists $y\in\cL^m_+$ with $w=M^\top y$ 
where $y_i=0$ for every $i$ that does not correspond to a row of $M^=$.
It follows from \eqref{CS-condition} that $y$ is an optimal solution of $(D\!:\!M,b,w)$,
a contradiction as $w$ is $\cL$-bad.
Thus, we may assume $F$ is a proper face of $\{x \in \R^n : Mx\leq b\}$.
Therefore, there exists a down-face $F^+$ of $F$.
By (ii) we have a solution of the $(w,F,F^+)$-tilt constraint with all variables in $\cL$.
However, this contradicts \Cref{extremal-key}.
\end{proof}
\noindent
We leave it as an exercise to check that the argument in the proof of \Cref{main-char} also shows that the condition (ii) in that theorem can be replaced by the following condition:
{\em ``For every admissible $w\in\cL^n$, denote by $F$ the set of optimal solutions of $(P\!:\!M,b,w)$.
For {\em some} down-face $F^+$ of $F$, the $(w,F,F^+)$-tilt constraint has a solution with variables in~$\cL$."}

\section*{Acknowledgment}
The authors are grateful for numerous discussions with Mahtab Alghasi.


\end{document}